\pdfoutput=1

\documentclass[12pt,a4paper,dvipsnames,x11names]{amsart}

\DeclareMathAlphabet{\mathpzc}{OT1}{pzc}{m}{it}

\usepackage[utf8]{inputenc}
\usepackage[T1]{fontenc}
\usepackage[english]{babel}
\usepackage[a4paper,margin=1in]{geometry}
\usepackage{microtype}
\usepackage{lmodern}
\usepackage{mathpazo}
\usepackage{euler}
\usepackage{amsmath,amssymb,amsthm,mathtools,mathrsfs,bm,eucal,tensor}
\usepackage{stmaryrd}
\usepackage{dsfont}
\usepackage[all,cmtip]{xy}
\usepackage[shortlabels]{enumitem}
\usepackage{braket,xspace,csquotes}
\usepackage[dvipsnames,x11names]{xcolor}
\usepackage{tikz}
\usetikzlibrary{arrows.meta,calc,positioning,fit,backgrounds,shapes.geometric}
\usepackage{tikz-cd}
\usepackage{hyperref}
\hypersetup{
  colorlinks=true,
  linkcolor=blue!60!black,
  citecolor=blue!60!black,
  urlcolor=blue!60!black
}
\usepackage[capitalize,nameinlink]{cleveref}
\usepackage{comment}

\tikzset{
  >=Stealth,
  box/.style={draw, rounded corners, inner sep=6pt},
  v/.style={draw, circle, inner sep=1.4pt},
  arr/.style={->, thick},
  darr/.style={->, thick, dashed},
  flow/.style={draw=black!45, rounded corners=8pt, fill=gray!6, inner sep=8pt, align=center},
  flowaccent/.style={draw=black!50, rounded corners=8pt, fill=gray!10, inner sep=8pt, align=center},
  tinylabel/.style={font=\footnotesize, inner sep=2pt}
}

\numberwithin{equation}{section}

\theoremstyle{plain}
\newtheorem{theorem}{Theorem}[section]
\newtheorem{proposition}[theorem]{Proposition}
\newtheorem{lemma}[theorem]{Lemma}
\newtheorem{corollary}[theorem]{Corollary}

\theoremstyle{definition}
\newtheorem{definition}[theorem]{Definition}

\crefname{lemma}{Lemma}{Lemmas}
\Crefname{lemma}{Lemma}{Lemmas}
\crefname{proposition}{Proposition}{Propositions}
\Crefname{proposition}{Proposition}{Propositions}
\crefname{theorem}{Theorem}{Theorems}
\Crefname{theorem}{Theorem}{Theorems}
\crefname{corollary}{Corollary}{Corollaries}
\Crefname{corollary}{Corollary}{Corollaries}
\crefname{remark}{Remark}{Remarks}
\Crefname{remark}{Remark}{Remarks}
\crefname{section}{Section}{Sections}
\Crefname{section}{Section}{Sections}

\newcommand{\C}{\mathbb C}

\newcommand{\cO}{\mathcal{O}}
\newcommand{\cF}{\mathcal{F}}
\newcommand{\cS}{\mathcal{S}}
\newcommand{\cM}{\mathcal{M}}

\newcommand{\mfM}{\mathfrak{M}}
\newcommand{\mfX}{\mathfrak{X}}
\newcommand{\bos}{\mathrm{bos}}

\newcommand{\NS}{\mathrm{NS}}

\newcommand{\PP}{\mathbf P}

\DeclareMathOperator{\id}{id}
\DeclareMathOperator{\Aut}{Aut}

\DeclareMathOperator{\Spec}{Spec}

\DeclareMathOperator{\pr}{pr}

\DeclareMathOperator{\Sym}{Sym}

\newcommand{\barmfM}{\overline{\mfM}}
\newcommand{\barcS}{\overline{\cS}}
\newcommand{\Msm}{\mfM}
\newcommand{\Ssm}{\cS}
\newcommand{\cHom}{\mathcal Hom}
\newcommand{\eps}{\varepsilon}
\newcommand{\Tr}{\operatorname{Tr}}

\title[Compactified supermoduli]{Compactified Supermoduli Space is Almost Never Projected}

\author{Mauricio Corr\^ea}
\address{Dipartimento di Matematica, Universit\`a degli Studi di Bari Aldo Moro, Bari, Italy}
\email{mauricio.barros@uniba.it}

\author{Simone Noja}
\address{Dipartimento di Matematica, Universit\`a degli Studi di Bari Aldo Moro, Bari, Italy}
\email{simone.noja@uniba.it}

\begin{document}

\begin{abstract}
We settle the projectedness problem for the compactified unpunctured
supermoduli stack in every genus at least two. In genus two, the odd component is split, whereas
the even component is non-projected. In every genus \(g\geq 3\), both compactified parity
components are non-projected.
\end{abstract}



\maketitle

\section{Introduction}

A super Riemann surface is a complex supermanifold built over an ordinary
Riemann surface and equipped with a distinguished odd direction in its tangent
bundle \cite{WittenNotes}.  This odd direction is not integrable in the usual sense: its square,
via the Lie bracket, generates the even tangent direction.  This is the
geometric incarnation of worldsheet supersymmetry.  
Over a point, such an object is equivalently an ordinary Riemann surface together with a theta
characteristic.  Accordingly, the moduli stack of super Riemann surfaces is a
supergeometric enhancement of the moduli stack of spin curves.

These objects are central in perturbative superstring theory.  Super Riemann
surfaces are the worldsheets of the theory, and their moduli superstacks are the
natural domains of integration for superstring measures.  
A basic structural question is whether supermoduli space admits a holomorphic
projection to its bosonic truncation.  If such a projection exists, then the odd
and even directions can be separated, at least formally, and integration over
supermoduli can be reduced to geometry on the moduli stack of spin curves.  The
work of Donagi--Witten showed that this expectation fails in general for the
open supermoduli spaces \cite{DWnp,DWat}.  For questions involving amplitudes,
degenerations, and the boundary behaviour of the superstring measure, however,
the natural objects are the Deligne--Mumford compactifications by stable
supercurves \cite{WittenNotes,FKP}.  The projectedness problem for these
compactified superstacks is therefore a more global form of the same
question.

The purpose of this paper is to determine projectedness for the compactified
unpunctured supermoduli stack in every genus \(g\geq2\) and in both parity
components.  The result is complete: the only projected compactified component
in genus at least two is the odd genus-two component, which is in fact split.
All other compactified components are non-projected, already at the level of the
primary Green obstruction.

\begin{theorem}[Main theorem]\label{thm:main}
Let \(\overline{\mathfrak M}_g^\pm\) denote the two parity components of the Deligne--Mumford
compactification of the moduli superstack of unpunctured super Riemann surfaces.
Then the following hold.
\begin{enumerate}
\item The odd compactified genus-two component \(\overline{\mathfrak M}_2^-\) is split.
\item The even compactified genus-two component \(\overline{\mathfrak M}_2^+\) is non-projected. More precisely,
\[
\omega_2(\overline{\mathfrak M}_2^+)\neq 0.
\]
\item For every \(g\geq 3\), both compactified parity components
\(\overline{\mathfrak M}_g^+\) and \(\overline{\mathfrak M}_g^-\) are non-projected. More precisely,
for every \(g\geq 3\),
\[
\omega_2(\overline{\mathfrak M}_g^\pm)\neq 0.
\]
\end{enumerate}
\end{theorem}

The second item was first proved by Felder--Kazhdan--Polishchuk in
\cite{FKP-boundary}.  Here we confirm their result by computing directly the primary Green obstruction
at the separating \((-,-)\)-boundary divisor. Our argument is not independent of \cite{FKP-boundary} though,  
as it relies on their boundary
expansion of the superperiod map and the elementary hyperelliptic
invariant-theoretic vanishing entering their analysis of projections. 

The proof of \Cref{thm:main} has three distinct parts.  In genus two and odd parity, the canonical
genus-two involution acts on the obstruction sheaf by a nontrivial character;
after passing to a finite \'etale cover, this forces the primary obstruction to
vanish and gives splitness.  In genus two and even parity, we instead compute
the normal principal part of the primary obstruction along the separating
\((-,-)\)-boundary divisor and show that it cannot be removed.  In all genera
\(g\geq3\), we propagate the Donagi--Witten one-pointed obstruction to the
compactified unpunctured moduli space by means of the elliptic-tail
Neveu--Schwarz boundary gluing map; the obstruction has a nonzero normal
component along this boundary and therefore survives in the ambient
compactified supermoduli stack.

The paper is organized as follows.  In \Cref{sec:primary-stack} we recall the
primary Green obstruction for Deligne--Mumford superstacks and establish the
finite \'etale and inertia criteria used later.  In \Cref{sec:g2-compactified}
we prove that the odd compactified genus-two component is split.  In
\Cref{sec:g2-even} we prove the non-projectedness of the even compactified
genus-two component by the boundary obstruction computation described above.
In \Cref{sec:DW-family} we recall the Donagi--Witten one-pointed obstruction
family.  In \Cref{sec:elliptic-tail} we construct the elliptic-tail
Neveu--Schwarz boundary map and compute the relevant tangent and normal
directions.  In \Cref{sec:survival} we prove that the one-pointed obstruction
survives under this gluing.  Finally, \Cref{sec:proof-main} assembles these
ingredients and proves \Cref{thm:main}.

\section{The primary obstruction on Deligne--Mumford superstacks}\label{sec:primary-stack}

We recall the form of Green's primary obstruction \cite{Green} used below.  The class is
obtained by the standard \'etale-local construction on smooth superschemes and
then descended along the associated \'etale groupoid.  For the obstruction theory
of supermanifolds and supermoduli we use Green's obstruction theory in the form
recalled by Donagi--Witten \cite{DWnp,DWat}. For more on obstruction theory, see also the recent \cite{CorNoj}. 
An account on superstacks is available in \cite{Bruzzo, CodViv}. 

\subsection{Splittings, projections, and the primary class}

Let \(\mfX\) be a smooth Deligne\allowbreak--Mumford superstack over \(\C\).  Set
\[
   X:=\mfX_{\bos}
\]
for its bosonic truncation.  Let \(J\subset\cO_{\mfX}\) be the ideal locally
generated by odd functions.  The odd conormal, or fermionic, sheaf is
\[
   \cF_{\mfX}:=J/J^2.
\]
Modulo the odd ideal, the tangent bundle $T \mfX$ decomposes as
\[
   T\mfX/JT\mfX=T_+\mfX\oplus T_-\mfX,
   \qquad
   T_-\mfX=\cF_{\mfX}^{\vee},
   \qquad
   T_+\mfX=T_X.
\]
Thus \(T_+\mfX\) and \(T_-\mfX\) are vector bundles on the ordinary
Deligne--Mumford stack \(X\).

\begin{definition}
A \emph{projection} of \(\mfX\) is a morphism of superstacks
\[
   \pi:\mfX\longrightarrow X
\]
which is a left inverse to the canonical closed embedding \(X\hookrightarrow\mfX\).
A \emph{splitting} of \(\mfX\) is an isomorphism of filtered superalgebras on
\(X\)
\[
   \cO_{\mfX}\simeq \bigwedge\nolimits^\bullet \cF_{\mfX}
\]
compatible with the augmentation to \(\cO_X\).
\end{definition}

A splitting induces a projection.  Conversely, in odd rank two, projectedness
and splitting are equivalent by Green's obstruction theory: there are no higher
Green obstructions beyond the primary obstruction, see for example \cite{CorNoj, DWnp}.

For clarity, we recall the local meaning of the coefficient sheaf.  On a split
chart, a change of local projection at second order is given by an even
derivation of \(\cO_X\) with values in \(\bigwedge^2\cF_{\mfX}\).  Equivalently,
it is a section of
\[
   T_X\otimes \bigwedge\nolimits^2T_-^\vee\mfX .
\]
Concretely, choose an \'etale cover \(\{U_i\}\) over which the first
infinitesimal neighbourhood is identified with the split model.  Local
second-order projections differ on overlaps by automorphisms of the form
\[
   1+q_{ij}
   \qquad
   q_{ij}\in\Gamma\bigl(U_{ij},
   T_X\otimes\wedge^2T_-^\vee\mfX\bigr).
\]
Modulo the third power of the odd ideal, the cocycle condition becomes the
additive identity
\[
   q_{ij}+q_{jk}=q_{ik}
   \qquad\text{on }U_{ijk},
\]
while changing the local second-order projections replaces \(q_{ij}\) by
\(q_{ij}+a_j-a_i\).  Thus the primary obstruction is the resulting \v{C}ech class
measuring the failure of local projections to glue at this order.

\begin{proposition}\label{prop:omega-stack}
Let \(\mfX\) be a smooth Deligne--Mumford superstack over \(\C\), and let
\(X=\mfX_{\bos}\).  There is a canonical primary obstruction class
\[
   \omega_2(\mfX)
   \in
   H^1\left(X,T_X\otimes\bigwedge\nolimits^2T_-^\vee\mfX\right).
\]
It is compatible with \'etale pullback.  If \(\omega_2(\mfX)\ne0\), then
\(\mfX\) is not projected.  If \(T_-\mfX\) has rank two, then
\[
   \mfX\text{ is split}\qquad\Longleftrightarrow\qquad\omega_2(\mfX)=0.
\]
\end{proposition}

\begin{proof}
Choose an \'etale atlas \(U\to\mfX\) by a smooth superscheme, and let
\[
   R:=U\times_{\mfX}U
\]
be the associated \'etale groupoid.  Applying Green's construction to \(U\) gives
a \v{C}ech representative of the primary obstruction on \(U_{\bos}\), with
coefficients in
\[
   T_{U_{\bos}}\otimes\bigwedge\nolimits^2T_{-,U}^\vee.
\]
The construction is functorial under \'etale morphisms.  Therefore the two
pullbacks of this class to \(R_{\bos}\) agree.  Equivalently, the local Green
classes define a descent datum for the \'etale groupoid
\[
   R_{\bos}\rightrightarrows U_{\bos}.
\]
Thus the obstruction is first obtained on the atlas and is then descended through
\[
\begin{tikzcd}[column sep=large]
   R_{\bos} \arrow[r,shift left=0.65ex,"s"] \arrow[r,shift right=0.65ex,"t"']
   & U_{\bos} \arrow[r] & X .
\end{tikzcd}
\]
The descended class lies in
\[
   H^1\left(X,T_X\otimes\bigwedge\nolimits^2T_-^\vee\mfX\right).
\]
The same functoriality gives compatibility with \'etale pullback.

The criterion that a non-zero primary obstruction prevents projectedness is
Green's obstruction-theoretic criterion, in the form recalled by
Donagi--Witten \cite[Corollaries 2.3--2.4]{DWnp}.  Since projectedness and
splitting are \'etale-local on the bosonic truncation and compatible with descent
along the \'etale groupoid, the same criterion holds for smooth
Deligne--Mumford superstacks.

Finally, when \(T_-\mfX\) has rank two, there are no odd exterior powers beyond
degree two and hence no higher Green obstructions.  Thus vanishing of the
primary obstruction is equivalent to projectedness, and in odd rank two
projectedness is equivalent to splitting.  This is the rank-two case of Green's
obstruction theory, again in the form recalled in \cite[Corollaries 2.3--2.4]{DWnp}.
\end{proof}

\subsection{Finite \'etale injectivity and an inertia criterion}

\begin{lemma}\label{lem:finite-etale-injective}
Let \(p:Y\to X\) be a finite \'etale surjective morphism of Deligne--Mumford
stacks over \(\C\), of constant degree \(d>0\).  Let \(\mathcal V\) be a vector
bundle on \(X\).  Then
\[
   p^*:H^i(X,\mathcal V)\longrightarrow H^i(Y,p^*\mathcal V)
\]
is injective for every \(i\ge0\).
\end{lemma}

\begin{proof}
Since \(p\) is finite \'etale of degree \(d\), there is a trace morphism
\[
   \Tr:p_*\cO_Y\longrightarrow\cO_X
\]
such that the composite
\[
   \cO_X\longrightarrow p_*\cO_Y\xrightarrow{\Tr}\cO_X
\]
is multiplication by \(d\).  Tensoring with \(\mathcal V\) and using the
projection formula gives a composite
\[
   \mathcal V\longrightarrow p_*p^*\mathcal V\xrightarrow{\Tr}\mathcal V
\]
which is multiplication by \(d\).  Since \(d\) is invertible over \(\C\),
\(\mathcal V\) is a direct summand of \(p_*p^*\mathcal V\).  Equivalently, one
has the split diagram
\[
\begin{tikzcd}[column sep=large]
   \mathcal V \arrow[r] \arrow[dr, "d\,\id_{\mathcal V}"']
   & p_*p^*\mathcal V \arrow[d, "\Tr"] \\
   & \mathcal V .
\end{tikzcd}
\]
Passing to cohomology gives the desired injectivity.
\end{proof}

\bigskip

\begin{definition}\label{def:inertia-stack}
Let \(\mfX\) be a Deligne--Mumford superstack, and let
\(
   \jmath:X:=\mfX_{\bos}\hookrightarrow \mfX
\)
be its bosonic truncation.  Here \(X=\mfX_{\bos}\) is the ordinary
Deligne--Mumford stack obtained, on an \'etale atlas, by quotienting the
structure sheaf by the ideal generated by odd sections.

The inertia stack of \(\mfX\) is
\[
   I_{\mfX}:=\mfX\times_{\mfX\times\mfX}\mfX .
\]
We define the \emph{bosonic inertia of \(\mfX\) along \(X\)} to be the ordinary
stack
\[
   I_{\mfX}^{\bos}\longrightarrow X
\]
whose \(T\)-points, for a bosonic test stack \(T\xrightarrow{f}X\), are
\[
   I_{\mfX}^{\bos}(T)
   =
   \Aut_{\mfX(T)}(\jmath\circ f).
\]
Equivalently,
\[
   I_{\mfX}^{\bos}
   \simeq
   \left(I_{\mfX}\times_{\mfX}X\right)_{\bos},
\]
where \(I_{\mfX}\to\mfX\) is the natural projection and the fibre product uses
\(\jmath:X\to\mfX\).

If \(p:Y\to X\) is a morphism from a bosonic Deligne--Mumford stack, a
\emph{bosonic inertia section over \(Y\)} is a section
\[
   a:Y\longrightarrow I_{\mfX}^{\bos}\times_XY .
\]
Thus \(a\) is a functorial automorphism, over the identity of \(Y\), of the
object of \(\mfX(Y)\) obtained from
\[
   Y\xrightarrow{p}X\xrightarrow{\jmath}\mfX .
\]

If \(\mfX\) is smooth, such an automorphism acts by its differential on the even
and odd tangent bundles of \(\mfX\) along \(X\), that is on $T_\pm \mfX$. Hence it acts on \(p^*T_X\), on
\(p^*T_-\mfX\), and on every vector bundle functorially constructed from them.
\end{definition}

Set for convenience
\(
   \mathcal V_{\mfX}:=T_X\otimes\bigwedge\nolimits^2T_-^\vee\mfX .
\)

\begin{lemma}[Inertia-character criterion]\label{lem:inertia-character}
Let \(\mfX\) be a smooth Deligne--Mumford superstack over \(\C\), and let
\(X=\mfX_{\bos}\).  Let \(p:Y\to X\) be finite \'etale and surjective.  Suppose
that \(Y\) carries a bosonic inertia section
\[
   a:Y\longrightarrow I_{\mfX}^{\bos}\times_XY
\]
whose induced action on \(p^*\mathcal V_{\mfX}\) is multiplication by a scalar
\(\chi\in\C^*\) with \(\chi\ne1\).  Then
\[
   \omega_2(\mfX)=0.
\]
\end{lemma}

\begin{proof}
The primary Green obstruction
\(
   \omega_2(\mfX)\in H^1(X,\mathcal V_{\mfX})
\)
is obtained by an \'etale-local construction and descends to the bosonic stack
\(X\).  In particular, it is natural under automorphisms of the objects
parametrized by \(\mfX\).  Therefore, after pulling back to \(Y\), the bosonic
inertia section \(a\) fixes the pulled-back obstruction class:
\[
   a\cdot p^*\omega_2(\mfX)=p^*\omega_2(\mfX).
\]
On the other hand, by assumption \(a\) acts on the coefficient bundle
\(p^*\mathcal V_{\mfX}\) by multiplication by \(\chi\).  Hence
\[
   a\cdot p^*\omega_2(\mfX)=\chi\,p^*\omega_2(\mfX).
\]
Since \(\chi\ne1\), it follows that \(p^*\omega_2(\mfX)=0\).  By
\Cref{lem:finite-etale-injective}, pullback along \(p\) is injective on
cohomology.  Therefore \(\omega_2(\mfX)=0\).
\end{proof}

\subsection{Functoriality under arbitrary morphisms}

We shall also use a functorial form of the Donagi--Witten compatibility lemma.

\begin{lemma}[Functoriality of the primary obstruction]\label{lem:functoriality}
Let \(f:\mathfrak Y\to\mathfrak X\) be a morphism of smooth Deligne--Mumford
superstacks.  Set \(X=\mathfrak X_{\bos}\), \(Y=\mathfrak Y_{\bos}\), and let
\(f_0:Y\to X\) be the induced morphism on bosonic truncations.  Let
\[
   df_+:T_+\mathfrak Y\longrightarrow f_0^*T_+\mathfrak X,
   \qquad
   df_-:T_-\mathfrak Y\longrightarrow f_0^*T_-\mathfrak X
\]
be the even and odd components of the differential, restricted to \(Y\).  Put
\[
   F_{\mathfrak X}:=\cHom(\wedge^2T_-\mathfrak X,T_+\mathfrak X),
   \qquad
   F_{\mathfrak Y}:=\cHom(\wedge^2T_-\mathfrak Y,T_+\mathfrak Y),
\]
and introduce the comparison sheaf on \(Y\)
\[
   G_f:=\cHom\bigl(\wedge^2T_-\mathfrak Y,f_0^*T_+\mathfrak X\bigr).
\]
The odd differential induces
\[
   \wedge^2df_-:\wedge^2T_-\mathfrak Y\longrightarrow f_0^*\wedge^2T_-\mathfrak X.
\]
Therefore pullback followed by precomposition with \(\wedge^2df_-\) gives a map
\[
   \iota_f:H^1(X,F_{\mathfrak X})\longrightarrow H^1(Y,G_f).
\]
Similarly, composition with the even differential gives a map
\[
   j_f:H^1(Y,F_{\mathfrak Y})\longrightarrow H^1(Y,G_f).
\]
Then
\[
   j_f\bigl(\omega_2(\mathfrak Y)\bigr)
   =
   \iota_f\bigl(\omega_2(\mathfrak X)\bigr)
   \quad\text{in }H^1(Y,G_f).
\]
In particular, if \(j_f(\omega_2(\mathfrak Y))\ne0\), then
\(\omega_2(\mathfrak X)\ne0\).
\end{lemma}

\begin{proof}
For ordinary supermanifolds, this is the Donagi--Witten functoriality statement for
the primary obstruction.  They first prove the compatibility formula for a
submanifold \(S'\subset S\), namely
\[
   j\bigl(\omega_2(S')\bigr)=\iota\bigl(\omega_2(S)\bigr),
\]
see \cite[Corollary 2.10]{DWnp}.  They then explain that the same construction
applies to an arbitrary morphism by using the two components of the differential;
this is \cite[Corollary 2.12]{DWnp}.

We pass from supermanifolds to Deligne--Mumford superstacks by \'etale descent.
Choose an \'etale atlas \(U\to\mathfrak X\) by a supermanifold and form the base
change \(\mathfrak Y_U=\mathfrak Y\times_{\mathfrak X}U\).  Choose an \'etale
atlas \(V\to\mathfrak Y_U\).  Then the morphism \(f\) is represented, after this
\'etale refinement, by a morphism of supermanifolds
\[
   f_U:V\longrightarrow U.
\]
By the supermanifold case,
\[
   j_{f_U}\bigl(\omega_2(V)\bigr)=\iota_{f_U}\bigl(\omega_2(U)\bigr).
\]
All the ingredients in this equality--the tangent bundles \(T_\pm\), the primary
obstruction, the differentials \(df_\pm\), and the coefficient morphisms
\(\iota\) and \(j\)--are compatible with \'etale pullback.  Hence the identity is
compatible with the \'etale groupoids presenting \(\mathfrak X\) and
\(\mathfrak Y\).  It therefore descends to \(H^1(Y,G_f)\).  The final
non-vanishing assertion is immediate.
\end{proof}

\section{The odd compactified genus-two component}\label{sec:g2-compactified}

We now introduce the main objects that will be used throughout this paper, and specialise the preceding formalism to the odd compactified genus-two
component of the supermoduli space. 

\subsection{Stable spin curves and the odd tangent bundle}

Felder--Kazhdan--Polishchuk prove that the moduli stack of stable supercurves
with prescribed Neveu--Schwarz and Ramond punctures is a smooth proper
Deligne--Mumford superstack \cite[Theorem A]{FKP}.  In the unpunctured case, the
bosonic truncation of \(\barmfM_g^\pm\) is the compactified spin stack
\(\barcS_g^\pm\).

Over an even base, a stable supercurve is described by a stable curve \(C\)
together with a generalized spin structure.  In the torsion-free model this is a
torsion-free rank-one sheaf \(L\) on \(C\) equipped with an isomorphism
\[
   L\xrightarrow{\sim}\cHom(L,\omega_C),
\]
where \(\omega_C\) is the dualising sheaf.  A node is Ramond when \(L\) is
locally free at the node and Neveu--Schwarz when it is not locally free; see
\cite{FKP}.

For a smooth spin curve \((C,L)\) of genus \(g\ge2\), with
\(L^{\otimes2}\simeq K_C\), the tangent spaces of the smooth supermoduli stack
are
\[
   T_+|_{(C,L)}=H^1(C,T_C),
   \qquad
   T_-|_{(C,L)}=H^1(C,L^{-1}).
\]
Equivalently, by Serre duality,
\[
   T_-^\vee|_{(C,L)}=H^0(C,K_C\otimes L).
\]
This follows from standard deformation theory of super Riemann surfaces, in
terms of the sheaf of superconformal vector fields; see \cite[Section 3.2]{DWnp}
and \cite[Section 2.2]{WittenNotes}.

In genus two, the stack \(\barmfM_2^-\) has dimension \(3|2\).  Let
\[
   E:=T_-\barmfM_2^-
\]
be its odd tangent bundle on \(\barcS_2^-\).  The primary obstruction lies in
\[
   H^1\left(\barcS_2^-,T_{\barcS_2^-}\otimes\det E^\vee\right).
\]
Thus Proposition \ref{prop:omega-stack} reduces the proof of splitness of
\(\barmfM_2^-\) to the vanishing of this class.

\subsection{The canonical genus-two involution}

Let \(\overline{\cM}_2\) be the ordinary Deligne--Mumford stack of stable curves
of genus two.  The bosonic truncation of \(\barmfM_2^-\) is
\[
   (\barmfM_2^-)_{\bos}=\barcS_2^-,
\]
and there is the usual forgetful morphism
\[
   \barcS_2^-\longrightarrow\overline{\cM}_2
\]
which forgets the spin structure.  The hyperelliptic involution on the smooth
genus-two locus extends canonically over the stable genus-two compactification.  Equivalently, the
universal stable genus-two curve carries a canonical involution extending the
hyperelliptic involution on the smooth locus.  We denote this involution by
\(\iota\).

\begin{proposition}\label{prop:hyperelliptic-invariance}
Let \((C,L,\kappa)\) be a stable odd genus-two spin curve, and let
\(\iota:C\to C\) be the canonical genus-two involution.  Then there exists an
isomorphism
\[
   \gamma:\iota^*L\xrightarrow{\sim}L
\]
compatible with the spin structure.
\end{proposition}

\begin{proof}
We first consider the smooth locus.  If \(C\) is a smooth genus-two curve, then
\(C\) is hyperelliptic.  The odd theta characteristics are precisely the line
bundles \(L\simeq\cO_C(w)\), where \(w\) is a Weierstrass point.  Since every
Weierstrass point is fixed by the hyperelliptic involution, \(\iota^*L\simeq L\)
on the smooth odd locus.

Now consider the compactified odd spin stack \(\barcS_2^-\).  The functor of
compatible lifts of \(\iota\) is a closed subfunctor of the relative isomorphism
functor between the two generalized spin structures \(\iota^*L\) and \(L\); the
compatibility condition with the spin form cuts out a finite closed subfunctor.
Explicitly, compatibility means that the square
\[
\begin{tikzcd}
   (\iota^*L)^{\otimes2} \arrow[r,"\gamma^{\otimes2}"] \arrow[d,"\iota^*\kappa"']
   & L^{\otimes2} \arrow[d,"\kappa"] \\
   \iota^*\omega_C \arrow[r,"\iota_\omega"'] & \omega_C
\end{tikzcd}
\]
commutes.  Hence the image of this finite closed functor in \(\barcS_2^-\) is a
closed substack.  This image contains the dense smooth odd spin locus.  Since
\(\barcS_2^-\) is irreducible, being the Cornalba compactification of the
irreducible odd spin moduli stack, the image is all of \(\barcS_2^-\).  Hence
every stable odd genus-two spin curve admits a compatible lift of the canonical
genus-two involution.
\end{proof}

\subsection{The smoothable lift cover}

Let \(T\) be a bosonic test scheme, or more generally a bosonic Deligne--Mumford
stack, and let \((C/T,L,\kappa)\) be a family of generalized stable spin curves
of genus two.  We denote by \(\iota_T:C\to C\) the relative canonical genus-two
involution, and by
\[
   \iota_\omega:\iota_T^*\omega_{C/T}\xrightarrow{\sim}\omega_{C/T}
\]
the induced isomorphism on the relative dualising sheaf.  A compatible lift of
\(\iota_T\) to the spin structure is an isomorphism
\[
   \gamma:\iota_T^*L\xrightarrow{\sim}L
\]
such that the spin structure is preserved.

\begin{definition}\label{def:lift-stack}
Let \(\mathscr L_2^-\to\barcS_2^-\) be the stack whose \(T\)-points are pairs
\[
   \bigl((C/T,L,\kappa),\gamma\bigr),
\]
where \((C/T,L,\kappa)\) is a \(T\)-point of \(\barcS_2^-\) and
\(\gamma:\iota_T^*L\xrightarrow{\sim}L\) is a compatible lift of the relative
canonical genus-two involution.  Morphisms in \(\mathscr L_2^-\) are isomorphisms
of stable spin curves commuting with the chosen compatible lifts.
\end{definition}

Over the open substack \(\cS_2^-\subset\barcS_2^-\) of smooth odd spin curves,
every compatible lift is smoothable, and
\(\mathscr L_2^-|_{\cS_2^-}\to\cS_2^-\) is the two-sheeted cover parametrising the
two compatible lifts.  We define
\[
   \widehat{\cS}_2^-\subset\mathscr L_2^-
\]
to be the stack-theoretic closure of this two-sheeted cover inside
\(\mathscr L_2^-\).  Equivalently, a geometric point of \(\widehat{\cS}_2^-\) is
a compatible lift which occurs as the limit of compatible lifts along a
smoothing to the smooth odd locus.  We call such lifts smoothable compatible
lifts.  We denote by
\[
   p:\widehat{\cS}_2^-\longrightarrow\barcS_2^-
\]
the morphism induced by forgetting the compatible lift.

The compatible lift will be used as a bosonic inertia section of
\(\barmfM_2^-\) over the finite \'etale cover
\[
   p:\widehat{\cS}_2^-\longrightarrow\barcS_2^-,
\]
in the sense of \Cref{def:inertia-stack}.

\begin{proposition}\label{prop:lift-cover}
The morphism
\[
   p:\widehat{\cS}_2^-\longrightarrow\barcS_2^-
\]
is finite \'etale, surjective, and of degree two.  Moreover, over
\(\widehat{\cS}_2^-\), the tautological compatible lift
\[
   \widehat\iota=(\iota,\gamma)
\]
defines a bosonic inertia section
\[
   \widehat\iota:
   \widehat{\cS}_2^-
   \longrightarrow
   I_{\barmfM_2^-}^{\bos}\times_{\barcS_2^-}\widehat{\cS}_2^- .
\]
\end{proposition}

\begin{proof}
The assertion is local on \(\barcS_2^-\).  We use Cornalba's quasistable
presentation of compactified spin curves \cite{Cornalba}; for the local
deformation coordinates we also use the standard description of the spin
smoothing parameters in Cornalba's model.

Over the smooth odd spin locus \(\cS_2^-\), the compatible lifts can be described
directly.  If \((C,L,\kappa)\) is a smooth odd genus-two spin curve, then
\(L\simeq\cO_C(w)\) for a Weierstrass point \(w\), and the hyperelliptic
involution fixes \(w\).  Hence \(\iota^*L\simeq L\).  Any compatible lift differs
from any other by the inessential automorphism \(\pm\id_L\).  Thus the
compatible lifts form a torsor under \(\mu_2\), and the lift stack is a finite
\'etale double cover over the smooth locus.

It remains to analyse the boundary.  Let \((X,\theta,b)\) be a quasistable spin
curve with stable model \(\beta:X\to C\).  Thus \(X\) is obtained from the
stable curve \(C\) by replacing some nodes by exceptional components
\(E\simeq\PP^1\), the line bundle \(\theta\) satisfies \(\theta|_E\simeq\cO_E(1)\)
on every exceptional component, and the associated torsion-free spin sheaf on
\(C\) is \(\beta_*\theta\).

Let \(X^{\mathrm{ne}}\subset X\) be the closed subcurve given by the union of the
non-exceptional irreducible components of \(X\), with its induced nodal curve
structure.  Denote by \(\operatorname{CC}(X^{\mathrm{ne}})\) the set of connected
components of \(X^{\mathrm{ne}}\).  An inessential automorphism is given by a
choice of signs
\[
   (\epsilon_Y)_{Y\in\operatorname{CC}(X^{\mathrm{ne}})}
   \in\{\pm1\}^{\operatorname{CC}(X^{\mathrm{ne}})},
\]
acting on \(\theta|_Y\) by multiplication by \(\epsilon_Y\).  The action on the
exceptional components is then determined by compatibility with the spin
structure.

Suppose that an exceptional component \(E_i\subset X\) meets the two
non-exceptional connected components \(Y_+\) and \(Y_-\) of \(X^{\mathrm{ne}}\),
possibly with \(Y_+=Y_-\).  Let \(t_i\) be the ordinary smoothing parameter of
the corresponding node of the stable model \(C\).  In the spin deformation
space, the corresponding coordinate is a square-root coordinate \(\tau_i\)
satisfying
\[
   t_i=\tau_i^2.
\]
With respect to this coordinate, the sign automorphism acts by
\[
   \tau_i\longmapsto \epsilon_{Y_+}\epsilon_{Y_-}\tau_i .
\]
The local configuration may be visualised as follows.
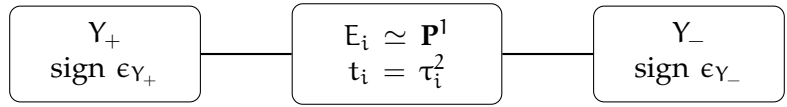
\[
\begin{tikzpicture}[node distance=1.2cm,every node/.style={font=\small}]
\node[box,text width=2.1cm,align=center] (yp) {$Y_+$\\ sign $\epsilon_{Y_+}$};
\node[box,text width=2.35cm,align=center,right=of yp] (ei) {$E_i\simeq\PP^1$\\ $t_i=\tau_i^2$};
\node[box,text width=2.1cm,align=center,right=of ei] (ym) {$Y_-$\\ sign $\epsilon_{Y_-}$};
\draw[-,thick] (yp) -- (ei);
\draw[-,thick] (ei) -- (ym);
\node[below=0.55cm of ei,text width=7.6cm,align=center]
{$\tau_i\mapsto\epsilon_{Y_+}\epsilon_{Y_-}\tau_i$; smoothability over the same base requires $\epsilon_{Y_+}=\epsilon_{Y_-}$.};
\end{tikzpicture}
\]

Fix one compatible lift of the canonical genus-two involution at the central
fibre.  Any other compatible lift differs from it by an inessential
automorphism, hence by a sign choice \((\epsilon_Y)_Y\) on the connected
components \(Y\) of \(X^{\mathrm{ne}}\).

Let \(\Gamma_X\) be the graph whose vertices are the connected components of
\(X^{\mathrm{ne}}\) and whose edges are the exceptional components of \(X\);
an exceptional component meeting the same connected component twice gives a
loop.  Since \(X\) is connected, \(\Gamma_X\) is connected.  A smoothing to the
smooth odd locus smooths every exceptional component.  For an exceptional
component \(E_i\) joining vertices \(Y_+\) and \(Y_-\), the corresponding spin
smoothing coordinate satisfies
\[
   \tau_i\longmapsto \epsilon_{Y_+}\epsilon_{Y_-}\tau_i .
\]
Thus a lift extends over the identity of the smoothing base if and only if
\[
   \epsilon_{Y_+}\epsilon_{Y_-}=1
\]
for every edge of \(\Gamma_X\).  Since \(\Gamma_X\) is connected, this is
equivalent to the signs being constant:
\[
   \epsilon_Y=\epsilon\quad\text{for all }Y,
   \qquad \epsilon\in\{\pm1\}.
\]
Conversely, the two constant sign choices act trivially on every spin smoothing
coordinate \(\tau_i\), and therefore extend over every smoothing family.  Hence
the smoothable compatible lifts are exactly the two constant-sign lifts.

In Cornalba's local deformation coordinates, the two smoothable lifts are
therefore represented by two disjoint sections of the finite lift stack over the
completed local ring at the given boundary point.  Since this description is
stable under \'etale localisation, the morphism \(p\) is \'etale-locally the
projection
\[
   \barcS_2^-\times\mu_2\longrightarrow \barcS_2^- .
\]
Equivalently, the local model is
\[
\begin{tikzcd}[column sep=large]
   \barcS_2^-\times\mu_2 \arrow[r, "\sim"] \arrow[d, "\pr_1"']
   & \widehat{\cS}_2^- \arrow[d, "p"] \\
   \barcS_2^- \arrow[r, equal]
   & \barcS_2^- .
\end{tikzcd}
\]
Hence \(p:\widehat{\cS}_2^-\to\barcS_2^-\) is finite \'etale, surjective, and of degree
two.

Finally, over \(\widehat{\cS}_2^-\) the compatible lift is tautological.  The pair
\(\widehat\iota=(\iota,\gamma)\) is therefore a functorial automorphism of the
pulled-back stable spin curve, equivalently of the corresponding stable
supercurve over the bosonic base \(\widehat{\cS}_2^-\).  Hence it defines a
bosonic inertia section
\[
   \widehat\iota:
   \widehat{\cS}_2^-
   \longrightarrow
   I_{\barmfM_2^-}^{\bos}\times_{\barcS_2^-}\widehat{\cS}_2^- .
\]
\end{proof}

\subsection{The inertia character and splitness}

Let \(E=T_-\barmfM_2^-\) be the odd tangent bundle on \(\barcS_2^-\), and put
\(\widehat E:=p^*E\).  Let \(\widehat\iota\) denote the tautological bosonic
inertia section of Proposition \ref{prop:lift-cover}.

\begin{proposition}\label{prop:character}
The bosonic inertia section \(\widehat\iota\) acts trivially on
$
   p^*T_{\barcS_2^-}
$
and acts by the scalar \(-1\) on
\(
   \det\widehat E^\vee.
\)
Consequently, it acts by the scalar \(-1\) on the primary obstruction sheaf
\[
   p^*\left(T_{\barcS_2^-}\otimes\det E^\vee\right).
\]
\end{proposition}

\begin{proof}
The smooth odd locus is dense in \(\widehat{\cS}_2^-\).  Since the objects
involved are vector bundles, it is enough to compute the character on this dense
open locus.
Let \((C,L)\) be a smooth genus-two odd spin curve.  Then \(C\) is hyperelliptic
and \(L\simeq\cO_C(w)\) for a Weierstrass point \(w\).  The even tangent space to
the spin stack at \((C,L)\) is \(H^1(C,T_C)\), whose dual is
\(H^0(C,K_C^{\otimes2})\).  The hyperelliptic involution acts by \(-1\) on
\(H^0(C,K_C)\).  Since in genus two the multiplication map
\[
   \Sym^2H^0(C,K_C)\xrightarrow{\sim}H^0(C,K_C^{\otimes2})
\]
is an isomorphism, the induced action on \(H^0(C,K_C^{\otimes2})\) is trivial.
Therefore the induced action on \(H^1(C,T_C)\) is trivial.  Hence
\(\widehat\iota\) acts trivially on \(p^*T_{\barcS_2^-}\).

By Serre duality, the odd cotangent space is
\[
   E^\vee|_{(C,L)}=H^0(C,K_C\otimes L).
\]
Since \(L\) is an odd theta characteristic of genus two, \(h^0(C,L)=1\).  Choose
a non-zero section \(s\in H^0(C,L)\).  Multiplication by \(s\) gives an
isomorphism
\[
   H^0(C,K_C)\otimes H^0(C,L)\xrightarrow{\sim}H^0(C,K_C\otimes L).
\]
Let \(\gamma:\iota^*L\xrightarrow{\sim}L\) be the chosen compatible lift.  It
induces an automorphism
\[
   A_\gamma:H^0(C,L)\longrightarrow H^0(C,L),
   \qquad
   A_\gamma(s)=\gamma(\iota^*s).
\]
Since \(H^0(C,L)\) is one-dimensional, \(A_\gamma(s)=\lambda s\) for some
\(\lambda\in\C^*\).  The compatibility of \(\gamma\) with the spin isomorphism
\(\kappa:L^{\otimes2}\simeq K_C\) implies that the induced action on the
non-zero differential \(\kappa(s^{\otimes2})\in H^0(C,K_C)\) is multiplication by
\(\lambda^2\).  On the other hand, the hyperelliptic involution acts by \(-1\) on
\(H^0(C,K_C)\).  Since \(\kappa(s^{\otimes2})\neq0\), it follows that
\[
   \lambda^2=-1.
\]
With respect to the multiplication isomorphism above, the action on the odd
cotangent space is
\[
   (-1)\otimes\lambda=-\lambda.
\]
Equivalently,
\[
\begin{tikzcd}[column sep=large,row sep=large]
   H^0(C,K_C)\otimes H^0(C,L) \arrow[r, "\sim"] \arrow[d, "(-1)\otimes\lambda"']
   & H^0(C,K_C\otimes L) \arrow[d, "-\lambda"] \\
   H^0(C,K_C)\otimes H^0(C,L) \arrow[r, "\sim"']
   & H^0(C,K_C\otimes L) .
\end{tikzcd}
\]
Since \(\dim H^0(C,K_C\otimes L)=2\), the action on its determinant is
\[
   (-\lambda)^2=\lambda^2=-1.
\]
Thus \(\widehat\iota\) acts by \(-1\) on \(\det E^\vee\).  This character is
independent of the choice of smoothable lift, since changing the lift replaces
\(\lambda\) by \(-\lambda\), while the determinant character remains
\(\lambda^2=-1\).  The character therefore extends to all of
\(\widehat{\cS}_2^-\).
\end{proof}

\begin{theorem}\label{thm:g2-main}
The odd compactified genus-two supermoduli stack \(\barmfM_2^-\) is split.
\end{theorem}

\begin{proof}
Let
\[
   \omega_2:=\omega_2(\barmfM_2^-)
   \in
   H^1\left(\barcS_2^-,
   T_{\barcS_2^-}\otimes\det E^\vee\right)
\]
be the primary obstruction class.  By \Cref{prop:lift-cover}, the finite
\'etale cover
\[
   p:\widehat{\cS}_2^-\longrightarrow\barcS_2^-
\]
carries the bosonic inertia section \(\widehat\iota\).  By
\Cref{prop:character}, this inertia section acts on
\[
   p^*\left(T_{\barcS_2^-}\otimes\det E^\vee\right)
\]
by the scalar \(-1\).  Therefore \Cref{lem:inertia-character} gives
\[
   \omega_2(\barmfM_2^-)=0.
\]
Since \(\barmfM_2^-\) has odd rank two, \Cref{prop:omega-stack} implies that
\(\barmfM_2^-\) is split.
\end{proof}

The argument is special to the odd spin component.  We next treat the even
component by a boundary computation of the primary obstruction.

\section{The even compactified genus-two component}\label{sec:g2-even}

We now prove the second assertion of \Cref{thm:main} by computing the primary
Green obstruction along the separating boundary divisor of type \((-,-)\).  The
point is to reinterpret the boundary pole found by Felder--Kazhdan--Polishchuk as
the non-zero normal principal part of the intrinsic class \(\omega_2\).

Set
\[
   \mathfrak X:=\barmfM_2^+,
   \qquad
   X:=\mathfrak X_{\bos}=\barcS_2^+,
   \qquad
   E:=T_-\mathfrak X .
\]
The primary obstruction is
\[
   \omega_2(\mathfrak X)
   \in
   H^1\left(X,\mathcal V\right),
   \qquad
   \mathcal V:=T_X\otimes\det E^\vee .
\]
Let
\[
   D:=D_{-,-}\subset X
\]
be the separating Neveu--Schwarz boundary divisor whose general point represents
an NS-nodal compact-type stable spin curve
\[
   (C_1,L_1,q_1)\cup (C_2,L_2,q_2),
\]
where \(C_1\) and \(C_2\) are smooth elliptic curves and \(L_1,L_2\) are odd
elliptic theta characteristics.  The total parity is even because
\(1+1\equiv 0\pmod 2\).

\subsection{The boundary principal part of the Green obstruction}

Put \(U:=X\setminus D\).  Felder--Kazhdan--Polishchuk prove that the canonical
projection coming from the superperiod map extends regularly to \(U\) and gives
a projection of \(\mathfrak X|_U\) onto \(U\)
\cite[Theorem~6.3(i)]{FKP-boundary}.  Denote this projection by \(\rho_U\).

Choose an \'etale chart \(V\to X\) meeting the generic point of \(D\), and complete
\(V\) along the generic point of \(D\).  We denote the resulting formal
neighbourhood by \(\widehat V\).  On \(\widehat V\) we use the standard separating
NS gluing coordinates
\[
   (\tau_1,\tau_2,t\mid \eta_1,\eta_2),
   \qquad
   D\cap \widehat V=(t=0).
\]
Here \(\tau_i\) are the elliptic moduli of the two components, \(t\) is the even
NS smoothing parameter, and \(\eta_i\) is the odd modulus of the \(i\)-th odd
one-NS-pointed elliptic component.  The formal coordinate projection
\[
   \rho_{\widehat V}^*(\tau_i)=\tau_i,
   \qquad
   \rho_{\widehat V}^*(t)=t
\]
is a regular local projection of the completed superstack
\(\mathfrak X|_{\widehat V}\) onto its bosonic formal neighbourhood.

Let \(j:U\hookrightarrow X\) be the inclusion.  The sheaf of principal parts of
\(\mathcal V\) along \(D\) is
\[
   \operatorname{PP}_D(\mathcal V):=j_*(\mathcal V|_U)/\mathcal V.
\]
Equivalently, at the generic point of \(D\), after choosing the local equation
\(t\), it records the negative powers of \(t\) in a meromorphic section of
\(\mathcal V\).  The normal principal part is obtained by projecting the
vector-field factor \(T_X\) to the normal line of \(D\).

On the punctured formal neighbourhood of the generic point of \(D\), the
difference between \(\rho_U\) and \(\rho_{\widehat V}\) is an
\(\eta_1\eta_2\)-valued vector field, hence a meromorphic section of
\(\mathcal V\).  We call its image in \(\operatorname{PP}_D(\mathcal V)\) the
boundary principal part of \(\omega_2(\mathfrak X)\) at \(D\).  The projection
\(\rho_U\) is fixed; replacing the formal coordinate projection
\(\rho_{\widehat V}\) by another regular projection on \(\widehat V\) changes the
comparison by a section regular along \(D\), and hence does not change this
principal part.

\begin{lemma}\label{lem:g2-even-principal-part}
At the generic point of \(D_{-,-}\), the image of the boundary principal part
modulo principal parts of vector fields tangent to \(D\) is
\[
   \operatorname{pp}^{\mathrm{nor}}_D\bigl(\omega_2(\mathfrak X)\bigr)
   =
   -\frac{(2\pi i)^{-2}}{2t^2}\,
   \partial_t\otimes(\eta_1\eta_2).
\]
In particular it is non-zero.
\end{lemma}

\begin{proof}
For an odd-rank-two supermanifold, the primary Green cocycle is represented by
pairwise differences of local projections.  Therefore the principal part of the
difference \(\rho_U-\rho_{\widehat V}\) gives the boundary principal part of the
obstruction class.  The plumbing computation of Felder--Kazhdan--Polishchuk at
\(D_{-,-}\) gives \cite[Corollary~5.10(ii)]{FKP-boundary}
\[
   \rho_U^*t
   =
   t-\frac{(2\pi i)^{-2}}{2t^2}\eta_1\eta_2+O(t^2)\eta_1\eta_2,
\]
where the omitted coefficients are regular at the generic point of \(D\).  Hence
on the punctured formal neighbourhood
\[
   \rho_U-\rho_{\widehat V}
   =
   -\frac{(2\pi i)^{-2}}{2t^2}\,
   \partial_t\otimes(\eta_1\eta_2)
   +\text{terms regular along }D
\]
modulo vector fields tangent to \(D\).  Passing to normal principal parts gives
the asserted formula.
\end{proof}

\subsection{The non-removability of the principal part}

We need the following invariant-theoretic vanishing, which is the genus-two
hyperelliptic calculation used by Felder--Kazhdan--Polishchuk in the proof of
the uniqueness statement for the canonical projection
\cite[proof of Theorem~6.3(ii)]{FKP-boundary}.  We recall it in the form needed
below.

\begin{lemma}\label{lem:g2-even-no-vector-field}
Let \(\xi\in H^0(U,\mathcal V|_U)\) be a section which is regular at the generic
points of the non-separating boundary divisors of \(X=\barcS_2^+\).  Then
\(\xi=0\).  Consequently no such section can have the principal part
\[
   c\,t^{-2}\partial_t\otimes(\eta_1\eta_2),
   \qquad c\ne0,
\]
along \(D_{-,-}\).
\end{lemma}

\begin{proof}
We work on the labelled hyperelliptic cover of the smooth even spin locus.  A
smooth genus-two curve with an even spin structure is written as
\[
   y^2=\prod_{i=1}^3(x-u_i)(x-v_i),
\]
where the two triples \((u_1,u_2,u_3)\) and \((v_1,v_2,v_3)\) determine the even
spin structure.  The relevant section must be invariant under the
\(\mathrm{SL}_2\)-action on \(x\), under permutations inside each triple, and
under interchange of the two triples.

Let
\[
   \left(\frac{dx}{y}\right)^2,\qquad
   x\left(\frac{dx}{y}\right)^2,\qquad
   x^2\left(\frac{dx}{y}\right)^2
\]
be the basis of \(H^0(C,K_C^{\otimes2})\) used in
\cite[Section~4.2]{FKP-boundary}, and let
\((\delta_0,\delta_1,\delta_2)\) be the dual tangent basis.  The odd cotangent
determinant is trivialized by \(\eta_1\eta_2\), where \((\eta_1,\eta_2)\) is the
basis dual to the two sections of \(H^0(C,K_C\otimes L)\) constructed in
\cite[(4.2)]{FKP-boundary}.  Hence, on the labelled cover, any section of
\(\mathcal V\) has the form
\[
   \xi=(f_0(u,v)\delta_0+f_1(u,v)\delta_1+f_2(u,v)\delta_2)\eta_1\eta_2 .
\]

Regularity at the generic non-separating boundary divisors means that no pole is
allowed when two branch points collide.  The only possible denominators in the
functions \(f_i\) are products of factors
\[
   (u_i-u_j),\qquad (v_i-v_j),\qquad (u_i-v_j).
\]
Thus the \(f_i\) are polynomials.

Now impose the \(\mathrm{SL}_2\)-equivariance.  Under the scaling
\[
   u_i\mapsto \lambda u_i,\qquad
   v_i\mapsto \lambda v_i,\qquad
   x\mapsto \lambda x,\qquad
   y\mapsto \lambda^3y,
\]
one has
\[
   \frac{dx}{y}\mapsto \lambda^{-2}\frac{dx}{y},
   \qquad
   \eta_1\eta_2\mapsto \lambda^{-3}\eta_1\eta_2.
\]
The dual vector fields \(\delta_0,\delta_1,\delta_2\) have weights \(4,3,2\),
respectively.  Hence the terms
\[
   \delta_0\eta_1\eta_2,\qquad
   \delta_1\eta_1\eta_2,\qquad
   \delta_2\eta_1\eta_2
\]
have weights \(1,0,-1\).  Invariance under scaling forces the polynomial
coefficients to have degrees
\[
   \deg f_0=-1,\qquad
   \deg f_1=0,\qquad
   \deg f_2=1.
\]
Thus \(f_0=0\), \(f_1\) is constant, and \(f_2\) is homogeneous of degree one.

Next use the inversion \(x\mapsto 1/x\).  This exchanges the first and third
quadratic differentials, and therefore exchanges the components
\(\delta_0\eta_1\eta_2\) and \(\delta_2\eta_1\eta_2\), up to a nowhere-zero
character factor as in \cite[proof of Theorem~6.3(ii)]{FKP-boundary}.  Since
\(f_0=0\), invariance under inversion forces \(f_2=0\).

Finally, the symmetry interchanging the two triples
\((u_1,u_2,u_3)\) and \((v_1,v_2,v_3)\) acts by \(-1\) on
\(\eta_1\eta_2\) and preserves the middle tangent component.  Hence the constant
\(f_1\) must satisfy \(f_1=-f_1\), so \(f_1=0\).  Therefore \(\xi=0\).
\end{proof}

\begin{theorem}\label{thm:g2-even-main}
The primary obstruction of the even compactified genus-two supermoduli stack is
non-zero:
\[
   \omega_2(\barmfM_2^+)\ne0.
\]
Consequently \(\barmfM_2^+\) is non-projected.
\end{theorem}

\begin{proof}
Suppose, to the contrary, that \(\omega_2(\mathfrak X)=0\).  Since
\(\mathfrak X\) has odd rank two, \Cref{prop:omega-stack} implies that
\(\mathfrak X\) is split; in particular, it admits a global projection
\[
   \rho:\mathfrak X\longrightarrow X.
\]
On \(U=X\setminus D\), the difference between the canonical projection
\(\rho_U\) and \(\rho\) is a section
\[
   \xi_U:=\rho_U-\rho\in H^0(U,\mathcal V|_U).
\]
This section is regular at the generic points of the non-separating boundary
divisors, because those divisors are contained in \(U\).  By
\Cref{lem:g2-even-no-vector-field}, \(\xi_U=0\).

On the formal neighbourhood \(\widehat V\) of the generic point of \(D\), the
difference \(\rho-\rho_{\widehat V}\) is regular along \(D\), since both
projections are defined on \(\widehat V\).  Therefore the normal principal part
of \(\rho_U-\rho_{\widehat V}\) along \(D\) is the same as the normal principal
part of \(\xi_U=\rho_U-\rho\).  This principal part is zero because
\(\xi_U=0\).  This contradicts \Cref{lem:g2-even-principal-part}, which gives the
non-zero normal principal part
\[
   -\frac{(2\pi i)^{-2}}{2t^2}\,\partial_t\otimes(\eta_1\eta_2).
\]
Hence \(\omega_2(\mathfrak X)\ne0\).

By \Cref{prop:omega-stack}, non-vanishing of the primary obstruction prevents
projectedness.  Therefore \(\barmfM_2^+\) is non-projected.
\end{proof}

\section{The Donagi--Witten one-pointed obstruction family}\label{sec:DW-family}

We now move to higher genera.  Before doing so, we recall the Donagi--Witten
one-pointed obstruction family \cite{DWnp}, which will be used in the proof of
the higher-genus statement.

Let \(h\ge2\).  Choose a smooth genus-\(h\) curve \(C\) and an ineffective even
theta characteristic
\[
   L^{\otimes2}\simeq\omega_C,
   \qquad
   H^0(C,L)=0.
\]
Such spin curves exist for every \(h\ge2\): a general even theta characteristic
is ineffective; see Mumford \cite{MumfordTheta} and Cornalba \cite{Cornalba}.
We use the Donagi--Witten notation
\[
   T_C^{1/2}:=L^{-1}.
\]
Since \(C\) is smooth and \(L^{\otimes2}\simeq\omega_C\), one has
\[
   T_C\simeq\omega_C^{-1}\simeq L^{-2}.
\]

The odd tangent space to the even supermoduli component at the split super
Riemann surface associated with \((C,L)\) is \(H^1(C,L^{-1})\).  For an odd
tangent vector \(\eta\in H^1(C,L^{-1})\), Donagi and Witten construct a
\((1|2)\)-dimensional supermanifold \(X_\eta\) by pulling back the universal one
Neveu--Schwarz pointed supercurve along the odd tangent line determined by
\(\eta\) \cite[Section 3.5]{DWnp}.  The reduced space of \(X_\eta\) is the curve
\(C\), regarded as the moving NS marking in the universal pointed curve over the
fixed spin curve \((C,L)\).  Thus \(X_\eta\) carries a natural morphism
\[
   u_\eta:X_\eta\longrightarrow\Msm_{h,1}^+.
\]

The odd part of the tangent bundle of \(X_\eta\) fits into an exact sequence
\[
   0\longrightarrow L^{-1}\longrightarrow T_-X_\eta
   \longrightarrow \cO_C\cdot\eta\longrightarrow0.
\]
After trivializing the odd base line by the chosen generator \(\eta\), we identify
\[
   \wedge^2T_-X_\eta\simeq L^{-1}.
\]
Moreover,
\[
   T_+X_\eta\simeq T_C\simeq L^{-2}.
\]
Hence
\[
   \cHom(\wedge^2T_-X_\eta,T_+X_\eta)
   \simeq
   \cHom(L^{-1},L^{-2})
   \simeq L^{-1}.
\]
Donagi--Witten prove that, under this identification,
\[
   \omega_2(X_\eta)=\eta\quad\text{in }H^1(C,L^{-1}),
\]
see \cite[Proposition 3.4]{DWnp}.

\section{The elliptic-tail Neveu--Schwarz boundary map}\label{sec:elliptic-tail}

Set \(g=h+1\).  Fix a smooth pointed elliptic spin curve \((E,M,q)\) with
\[
   M^{\otimes2}\simeq\omega_E\simeq\cO_E.
\]
Let \(\eps\in\{+,-\}\) be the parity of \(M\).  For compact-type nodal spin
curves, parity is the sum of the parities of the restrictions to the components.
Since the spin curve \((C,L)\) chosen above is even, the parity of the glued
curve is the parity of \(M\).  On a smooth elliptic curve, the trivial theta
characteristic is odd, whereas the three nontrivial two-torsion theta
characteristics are even.

We shall regard
\[
   \xi_{\bos}:=(E,M,q)\in \Ssm_{1,1}^{\eps}(\C)
\]
as a stack-theoretic point of the bosonic spin stack, and we denote by
\[
   \xi\in \Msm_{1,1}^{\eps}(\C)
\]
the corresponding split one-NS-pointed elliptic supercurve.  Let
\[
   A_\xi:=\Aut_{\Msm_{1,1}^{\eps}}(\xi)
\]
be its automorphism group.  Since \(\Msm_{1,1}^{\eps}\) is Deligne--Mumford,
\(A_\xi\) is finite \'etale.  Let
\[
   \mathcal G_\xi\longrightarrow \Msm_{1,1}^{\eps}
\]
be the residual gerbe of the point \(\xi\); thus
\(\mathcal G_\xi\simeq \mathrm B A_\xi\).  In what follows, saying that the
elliptic tail is fixed means that we use the object \(\xi\) as a constant family
over every test superscheme.  (It does not mean that we replace the elliptic
factor by an ordinary representable point.)

\smallskip

Consider the product of smooth one-NS-pointed supermoduli stacks
\[
   B^\eps:=\Msm_{h,1}^+\times\Msm_{1,1}^\eps.
\]
Its bosonic truncation is
\[
   B^\eps_{\bos}=\Ssm_{h,1}^+\times\Ssm_{1,1}^\eps.
\]
The stable-supercurve compactification admits the NS gluing morphism
\[
   \Gamma^\eps:B^\eps\longrightarrow\barmfM_{h+1}^\eps,
\]
which glues the two NS punctures and produces a separating NS node.  We denote
the induced morphism on bosonic truncations by
\[
   \Gamma^\eps_{\bos}:B^\eps_{\bos}\longrightarrow\barcS_{h+1}^\eps.
\]
By Felder--Kazhdan--Polishchuk, the NS gluing morphism is defined for stable
supercurves with NS punctures; near a curve with a single NS node, its image is
the corresponding NS boundary component, and it is a codimension \(1|0\)
embedding on tangent spaces \cite[Lemmas 7.9--7.10]{FKP}.  With their convention
for the NS psi-lines, the normal line to the NS boundary is
\[
   \cO(\Delta_{\NS})|_{B^\eps}
   \simeq
   \Psi_1^{-1}\otimes\Psi_2^{-1},
\]
where \(\Psi_i\) is the NS psi-line attached to the \(i\)-th puncture; see
\cite[Theorem C and equation (9.6)]{FKP}.

\smallskip

We now define carefully the fixed-tail construction as a morphism of
superstacks.  For every test superscheme \(S\) and every \(S\)-point
\(\mathcal X\in\Msm_{h,1}^+(S)\), set
\[
   \xi_S:=\xi\times_{\Spec\C} S.
\]
The assignment
\[
   \mathcal X\longmapsto (\mathcal X,\xi_S)
\]
and, on morphisms, \(\alpha\mapsto(\alpha,\id_{\xi_S})\), defines a morphism
of superstacks
\[
   \kappa_\xi:\Msm_{h,1}^+\longrightarrow
   B^\eps=\Msm_{h,1}^+\times\Msm_{1,1}^{\eps}.
\]
Equivalently, \(\kappa_\xi\) is the constant-tail morphism determined by the
chosen object \(\xi\).

The fixed-tail substack of \(B^\eps\) which retains all automorphisms of the
elliptic tail is
\[
   \Msm_{h,1}^+\times \mathcal G_\xi
   \longrightarrow
   \Msm_{h,1}^+\times\Msm_{1,1}^{\eps}.
\]
The morphism \(\kappa_\xi\) is the map determined by the chosen point
\(\Spec\C\to\mathcal G_\xi\).  We will not use representability of
\(\kappa_\xi\).  The extra automorphisms of the elliptic tail remain as inertia
in the target stack; they do not quotient the source of the test family.

Define the fixed-tail gluing morphism
\[
   G_{E,M,q}:=\Gamma^\eps\circ\kappa_\xi:
   \Msm_{h,1}^+\longrightarrow\barmfM_{h+1}^{\eps}.
\]
We define the test morphism
\[
   F_{\eta,M}:=G_{E,M,q}\circ u_\eta:
   X_\eta\longrightarrow\barmfM_{h+1}^{\eps}.
\]

On bosonic truncations, the corresponding constant-tail morphism is
\[
   \kappa_{\xi,\bos}:\Ssm_{h,1}^+
   \longrightarrow
   B^\eps_{\bos}=\Ssm_{h,1}^+\times\Ssm_{1,1}^{\eps}.
\]
Let
\[
   e_C:C\longrightarrow \Ssm_{h,1}^+,
   \qquad
   p\longmapsto (C,L,p),
\]
and define
\[
   e:=\kappa_{\xi,\bos}\circ e_C:
   C\longrightarrow B^\eps_{\bos},
   \qquad
   p\longmapsto\bigl((C,L,p),(E,M,q)\bigr).
\]
Set
\[
   i:=\Gamma^\eps_{\bos}\circ e:C\longrightarrow\barcS_{h+1}^{\eps}.
\]
Thus
\[
   i(p)=(C,L,p)\cup_{p=q}(E,M,q)
\]
is the stable spin curve obtained by gluing the moving point \(p\in C\) to the
fixed point \(q\in E\) by an NS node.  Schematically, the family is
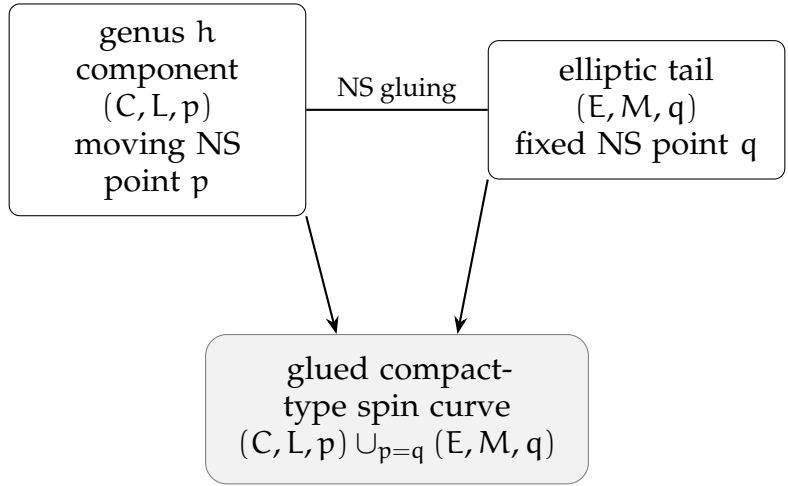
\[
\begin{tikzpicture}[baseline=(current bounding box.center)]
\node[box, text width=3.5cm, align=center] (C) {genus \(h\) component\\ \((C,L,p)\)\\ moving NS point \(p\)};
\node[box, text width=3.5cm, align=center, right=24mm of C] (E) {elliptic tail\\ \((E,M,q)\)\\ fixed NS point \(q\)};
\node[flowaccent, text width=4.5cm, align=center, below=18mm of $(C.south)!0.5!(E.south)$] (G) {glued compact-type spin curve\\ \((C,L,p)\cup_{p=q}(E,M,q)\)};
\draw[-, thick] (C.east) -- node[tinylabel, above] {NS gluing} (E.west);
\draw[arr] (C.south east) -- ([xshift=-8mm]G.north);
\draw[arr] (E.south west) -- ([xshift=8mm]G.north);
\end{tikzpicture}
\]
All coefficient sheaves used below are pulled back along \(e\) or \(i\) to the
ordinary curve \(C\).  Thus the calculation is not a cohomology calculation over
the residual gerbe \(\mathcal G_\xi\); no invariant part under \(A_\xi\) is being
taken.

The purpose of this section is to compute the even tangent quotient along this
test curve.  Here \(T_C\) denotes the vertical tangent of the universal pointed
curve \(\Ssm_{h,1}^+\to\Ssm_h^+\); equivalently, it is the tangent direction
obtained by moving the NS marking while keeping the spin curve \((C,L)\) fixed.

\begin{lemma}\label{lem:tangent-source}
Along the morphism \(e:C\to B^\eps_{\bos}\), there is an exact sequence
\[
   0\longrightarrow T_C
   \longrightarrow e^*TB^\eps_{\bos}
   \longrightarrow W\otimes\cO_C
   \longrightarrow0,
\]
where
\[
   W:=H^1(C,T_C)\oplus T_{\xi_{\bos}}\Ssm_{1,1}^{\eps}
\]
is a fixed finite-dimensional vector space.  More explicitly,
\[
   e^*TB^\eps_{\bos}/T_C
   \simeq
   \bigl(H^1(C,T_C)\otimes\cO_C\bigr)
   \oplus
   \bigl(T_{\xi_{\bos}}\Ssm_{1,1}^{\eps}\otimes\cO_C\bigr).
\]
\end{lemma}

\begin{proof}
All assertions concern vector bundles pulled back to the ordinary curve \(C\) and
may be checked on \'etale charts.  Let
\[
   e_C:C\longrightarrow\Ssm_{h,1}^+,
   \qquad
   p\longmapsto(C,L,p).
\]
The morphism \(\Ssm_{h,1}^+\to\Ssm_h^+\) is the universal pointed curve over the
smooth spin moduli stack.  Pulling back its tangent sequence along \(e_C\) gives
\[
   0\longrightarrow T_C
   \longrightarrow e_C^*T\Ssm_{h,1}^+
   \longrightarrow T_{(C,L)}\Ssm_h^+\otimes\cO_C
   \longrightarrow0.
\]
Since the smooth spin moduli stack is finite \'etale over the ordinary moduli stack
of smooth curves,
\[
   T_{(C,L)}\Ssm_h^+\simeq H^1(C,T_C).
\]

For the elliptic factor, the morphism \(e\) is constant and equal to the
stack-theoretic point \(\xi_{\bos}\).  Therefore
\[
   e^*\pr_2^*T\Ssm_{1,1}^{\eps}
   \simeq
   T_{\xi_{\bos}}\Ssm_{1,1}^{\eps}\otimes\cO_C.
\]
If this statement is expressed through the residual gerbe
\(\mathcal G_{\xi,\bos}\), then the vector space
\(T_{\xi_{\bos}}\Ssm_{1,1}^{\eps}\) is naturally an \(A_\xi\)-representation.
Pulling it back along the chosen point \(\Spec\C\to\mathcal G_{\xi,\bos}\)
gives the underlying constant vector bundle on \(C\).  Since \(A_\xi\) is finite
\'etale, the gerbe itself contributes no infinitesimal tangent directions.

Combining the two factors in
\[
   B^\eps_{\bos}=\Ssm_{h,1}^+\times\Ssm_{1,1}^{\eps}
\]
gives the stated exact sequence and the displayed description of the quotient.
\end{proof}

\begin{lemma}\label{lem:normal-line}
The normal line to the NS boundary gluing morphism restricts along
\(e:C\to B^\eps_{\bos}\) as
\[
   e^*N_\Gamma\simeq L^{-1}\otimes M_q^{-1},
\]
where \(M_q:=M|_q\).
\end{lemma}

\begin{proof}
Let \(N_\Gamma\) denote the normal line to the NS boundary component along the
image of the gluing morphism.  By the Felder--Kazhdan--Polishchuk normal-line
formula,
\[
   N_\Gamma\simeq\Psi_1^{-1}\otimes\Psi_2^{-1}.
\]
Over an even base, the inverse NS psi-line at a marked point is identified with
the fibre of the inverse spin line at that point \cite[Example 8.9]{FKP}.  In
the present situation the first marked point is the moving point \(p\in C\),
whereas the second marked point is the fixed point \(q\in E\).  Hence
\[
   e^*\Psi_1^{-1}\simeq L^{-1},
   \qquad
   e^*\Psi_2^{-1}\simeq M_q^{-1}\otimes\cO_C.
\]
Therefore \(e^*N_\Gamma\simeq L^{-1}\otimes M_q^{-1}\).  Stack-theoretically,
\(M_q^{-1}\) is the one-dimensional representation obtained by evaluating the
fixed spin line \(M^{-1}\) at \(q\); pulling back along the chosen point of the
residual gerbe gives the ordinary constant line
\(M_q^{-1}\otimes\cO_C\). (Thus this calculation, like the tangent calculation
above, is performed after pullback to \(C\), not after taking invariants over
the tail automorphism group.)  The individual
NS psi-line factors carry odd parity in the super tangent sequence, but their
tensor product is even; it is this even tensor product which gives the ordinary
normal line to the bosonic boundary divisor.
\end{proof}

\begin{proposition}[Even tangent quotient along the test curve]\label{prop:tangent-quotient}
Let
\[
   T:=i^*T_+\barmfM_{h+1}^{\eps}\simeq i^*T\barcS_{h+1}^{\eps}.
\]
The differential of \(i\) gives an injection \(T_C\hookrightarrow T\).  If
\(Q:=T/T_C\), then \(Q\) has a filtration
\[
   0\subset Q_{\mathrm{main}}\subset Q_{\mathrm{bd}}\subset Q
\]
whose associated graded pieces are
\[
   Q_{\mathrm{main}}
   \simeq
   H^1(C,T_C)\otimes\cO_C,
\]
\[
   Q_{\mathrm{bd}}/Q_{\mathrm{main}}
   \simeq
   T_{\xi_{\bos}}\Ssm_{1,1}^{\eps}\otimes\cO_C,
\]
\[
   Q/Q_{\mathrm{bd}}
   \simeq
   L^{-1}\otimes M_q^{-1}.
\]
Equivalently, if
\[
   W:=H^1(C,T_C)\oplus T_{\xi_{\bos}}\Ssm_{1,1}^{\eps},
\]
then \(Q\) has a two-step filtration
\[
   0\subset Q_0\subset Q
\]
with
\[
   Q_0\simeq W\otimes\cO_C,
   \qquad
   Q/Q_0\simeq L^{-1}\otimes M_q^{-1}.
\]
\end{proposition}

\begin{proof}
By Lemma \ref{lem:tangent-source}, the quotient of \(e^*TB^\eps_{\bos}\) by the
moving tangent direction \(T_C\) is
\[
   e^*TB^\eps_{\bos}/T_C
   \simeq
   \bigl(H^1(C,T_C)\otimes\cO_C\bigr)
   \oplus
   \bigl(T_{\xi_{\bos}}\Ssm_{1,1}^{\eps}\otimes\cO_C\bigr).
\]
By the tangent-embedding property of the NS boundary gluing morphism, pulling
back the tangent sequence of the boundary inclusion along \(e\) gives
\[
   0\longrightarrow e^*TB^\eps_{\bos}
   \longrightarrow T
   \longrightarrow e^*N_\Gamma
   \longrightarrow0.
\]
Together with the quotient sequence defining \(Q=T/T_C\), these exact sequences
fit into the following commutative diagram:
\[
\begin{tikzcd}[column sep=large,row sep=large]
  & 0 \arrow[d] & 0 \arrow[d] & 0 \arrow[d] & \\
0 \arrow[r] &
T_C \arrow[r] \arrow[d,equal] &
e^*TB^\eps_{\bos} \arrow[r] \arrow[d,hook] &
e^*TB^\eps_{\bos}/T_C \arrow[r] \arrow[d,hook] &
0 \\
0 \arrow[r] &
T_C \arrow[r] \arrow[d] &
T \arrow[r] \arrow[d] &
Q \arrow[r] \arrow[d] &
0 \\
0 \arrow[r] &
0 \arrow[r] &
e^*N_\Gamma \arrow[r,equal] \arrow[d] &
e^*N_\Gamma \arrow[r] \arrow[d] &
0 \\
& & 0 & 0 &
\end{tikzcd}
\]
Let \(Q_{\mathrm{bd}}\subset Q\) be the image of
\(e^*TB^\eps_{\bos}/T_C\).  The diagram gives an exact sequence
\[
   0\longrightarrow Q_{\mathrm{bd}}
   \longrightarrow Q
   \longrightarrow e^*N_\Gamma
   \longrightarrow0.
\]
By Lemma \ref{lem:normal-line},
\[
   e^*N_\Gamma\simeq L^{-1}\otimes M_q^{-1}.
\]

It remains to record the internal structure of \(Q_{\mathrm{bd}}\).  Under the
direct-sum identification above, define \(Q_{\mathrm{main}}\subset Q_{\mathrm{bd}}\)
as the image of the summand
\[
   H^1(C,T_C)\otimes\cO_C
   \subset e^*TB^\eps_{\bos}/T_C.
\]
Then
\[
   Q_{\mathrm{main}}
   \simeq
   H^1(C,T_C)\otimes\cO_C,
\]
and
\[
   Q_{\mathrm{bd}}/Q_{\mathrm{main}}
   \simeq
   T_{\xi_{\bos}}\Ssm_{1,1}^{\eps}\otimes\cO_C.
\]
Thus \(Q\) has the asserted three-step filtration.  Grouping the two constant
summands gives the equivalent two-step filtration
\[
   0\subset Q_0\subset Q,
   \qquad
   Q_0\simeq W\otimes\cO_C,
   \qquad
   Q/Q_0\simeq L^{-1}\otimes M_q^{-1},
\]
where
\(
   W:=H^1(C,T_C)\oplus T_{\xi_{\bos}}\Ssm_{1,1}^{\eps}.
\)
\end{proof}

\section{The survival of the obstruction under elliptic-tail gluing}\label{sec:survival}

Let
\[
   J:H^1(C,L^{-1})\longrightarrow H^1\bigl(C,\cHom(L^{-1},T)\bigr)
\]
be the map induced by the inclusion \(T_C\simeq L^{-2}\hookrightarrow T\).  Since
\(\cHom(L^{-1},T)=T\otimes L\), this is equivalently the map induced by
\[
   \cHom(L^{-1},T_C)\simeq L^{-1}
   \longrightarrow
   \cHom(L^{-1},T)\simeq T\otimes L.
\]

\begin{lemma} \label{lem:kernel}
With the notation above,
$
   \dim\ker J\le1.
$
\end{lemma}

\begin{proof}
Apply \(\cHom(L^{-1},-)=(-)\otimes L\) to the exact sequence
\[
   0\longrightarrow T_C\longrightarrow T\longrightarrow Q\longrightarrow0.
\]
Using \(T_C\simeq L^{-2}\), we obtain
\[
   0\longrightarrow L^{-1}
   \longrightarrow T\otimes L
   \longrightarrow Q\otimes L
   \longrightarrow0.
\]
The associated long exact sequence in cohomology contains
\[
   H^0(C,Q\otimes L)
   \xrightarrow{\delta}
   H^1(C,L^{-1})
   \xrightarrow{J}
   H^1(C,T\otimes L).
\]
Hence \(\ker J=\operatorname{im}\delta\), and therefore
\[
   \dim\ker J\le h^0(C,Q\otimes L).
\]
By Proposition \ref{prop:tangent-quotient}, after tensoring by \(L\) the quotient
\(Q\otimes L\) has a filtration whose associated graded terms are
\[
   Q_0\otimes L\simeq W\otimes L,
   \qquad
   (Q/Q_0)\otimes L\simeq \cO_C\otimes M_q^{-1}.
\]
Equivalently, there is an exact sequence
\[
   0\longrightarrow W\otimes L
   \longrightarrow Q\otimes L
   \longrightarrow \cO_C\otimes M_q^{-1}
   \longrightarrow0.
\]
Taking global sections gives an injection
\[
   H^0(C,Q\otimes L)
   \hookrightarrow
   H^0(C,\cO_C)\otimes M_q^{-1}
\]
because \(H^0(C,W\otimes L)=W\otimes H^0(C,L)=0\).  Hence
\[
   h^0(C,Q\otimes L)\le h^0(C,\cO_C)=1.
\]
Therefore \(\dim\ker J\le1\).
\end{proof}

\begin{corollary}\label{cor:eta-survives}
There exists \(\eta\in H^1(C,L^{-1})\) such that \(J(\eta)\ne0\).
\end{corollary}

\begin{proof}
Since \(L^{\otimes2}\simeq\omega_C\), we have
\[
   \deg L=h-1,
   \qquad
   \deg L^{-1}=1-h<0.
\]
Thus \(H^0(C,L^{-1})=0\).  By Riemann--Roch,
\[
   \chi(C,L^{-1})=h^0(C,L^{-1})-h^1(C,L^{-1})
   =\deg(L^{-1})+1-h=2-2h.
\]
Since \(H^0(C,L^{-1})=0\), this gives
\[
   h^1(C,L^{-1})=2h-2.
\]
For \(h\ge2\), the source of \(J\) has dimension at least \(2\).  By Lemma
\ref{lem:kernel}, \(\ker J\) has dimension at most \(1\), and hence it cannot be all of
\(H^1(C,L^{-1})\).  Choose \(\eta\notin\ker J\).
\end{proof}

\section{Proof of the main theorem}\label{sec:proof-main}

We now assemble the two arguments established in the preceding sections.  The
genus-two assertions were already proved in
\Cref{sec:g2-compactified} and \Cref{sec:g2-even}. The higher-genus assertion follows from the
Donagi--Witten test family, the elliptic-tail gluing construction, and the
functoriality of the primary obstruction.

\begin{proof}[Proof of \Cref{thm:main}]
The first assertion is precisely \Cref{thm:g2-main}: the odd compactified
genus-two component \(\barmfM_2^-\) is split.  The second assertion is the
Felder--Kazhdan--Polishchuk genus-two even non-projectedness theorem, recalled
in the form needed here in \Cref{sec:g2-even}.  In the notation of this paper,
this is \Cref{thm:g2-even-main}: the primary obstruction of
\(\barmfM_2^+\) is non-zero, and hence \(\barmfM_2^+\) is non-projected.

\smallskip

We now prove the higher-genus assertion.  Fix
\(\eps\in\{+,-\}\).  Let \(g\ge 3\), and put
\[
        h:=g-1\ge 2.
\]
Choose a smooth curve \(C\) of genus \(h\), together with an ineffective even
theta characteristic
\[
        L^{\otimes 2}\simeq \omega_C,
        \qquad
        H^0(C,L)=0.
\]
Choose a pointed elliptic spin curve \((E,M,q)\) of parity \(\eps\).  Such a
choice exists for both parities: the trivial theta characteristic on \(E\) is
odd, whereas the three nontrivial two-torsion theta characteristics are even.

For this choice of elliptic tail, the construction of
\Cref{sec:elliptic-tail} gives the map
\[
        J:
        H^1(C,L^{-1})
        \longrightarrow
        H^1\bigl(C,\cHom(L^{-1},T)\bigr).
\]
By \Cref{cor:eta-survives}, there exists a class
\[
        \eta\in H^1(C,L^{-1})
\]
such that
\[
        J(\eta)\ne 0.
\]

Let \(X_\eta\) be the Donagi--Witten test supermanifold associated with
\(\eta\).  By the construction recalled in \Cref{sec:DW-family}, its primary
obstruction is
\[
        \omega_2(X_\eta)=\eta
        \in  H^1(C,L^{-1}),
\]
under the identifications
\[
        T_+X_\eta\simeq T_C\simeq L^{-2},
        \qquad
        \bigwedge\nolimits^2T_-X_\eta\simeq L^{-1}.
\]
The constant-tail NS gluing construction of \Cref{sec:elliptic-tail} gives a
morphism of superstacks
\[
        G_{E,M,q}:\Msm_{h,1}^+\longrightarrow \barmfM_g^\eps .
\]
Composing it with the Donagi--Witten morphism
\[
        u_\eta:X_\eta\longrightarrow \Msm_{h,1}^+
\]
gives the test morphism
\[
        F_{\eta,M}:=G_{E,M,q}\circ u_\eta:
        X_\eta\longrightarrow \barmfM_g^\eps .
\]
Thus we have the commutative diagram
\[
\begin{tikzcd}[column sep=large,row sep=large]
        X_\eta
        \arrow[r,"u_\eta"]
        \arrow[dr,"F_{\eta,M}"']
        &
        \Msm_{h,1}^+
        \arrow[d,"G_{E,M,q}"]
        \\
        &
        \barmfM_g^\eps .
\end{tikzcd}
\]

On bosonic truncations, the morphism \(F_{\eta,M}\) restricts to the boundary
test curve
\[
        i:C\longrightarrow \barcS_g^\eps,
        \qquad
        p\longmapsto (C,L,p)\cup_{p=q}(E,M,q).
\]
Equivalently, if
\[
        e:C\longrightarrow \Ssm_{h,1}^+\times \Ssm_{1,1}^\eps,
        \qquad
        p\longmapsto \bigl((C,L,p),(E,M,q)\bigr),
\]
then $
        i=\Gamma_{\bos}^{\eps}\circ e.$
This is expressed by the diagram
\[
\begin{tikzcd}[column sep=large,row sep=large]
        C
        \arrow[r,"e"]
        \arrow[dr,"i"']
        &
        \Ssm_{h,1}^+\times \Ssm_{1,1}^\eps
        \arrow[d,"\Gamma_{\bos}^{\eps}"]
        \\
        &
        \barcS_g^\eps .
\end{tikzcd}
\]

Set
$
        T:=i^*T_+\barmfM_g^\eps .
$
For the morphism \(F_{\eta,M}\), the comparison sheaf in
\Cref{lem:functoriality} is
\[
        G_{F_{\eta,M}}
        :=
        \cHom\!\left(
        \bigwedge\nolimits^2T_-X_\eta,
        i^*T_+\barmfM_g^\eps
        \right).
\]
Using
\[
        \bigwedge\nolimits^2T_-X_\eta\simeq L^{-1},
        \qquad
        i^*T_+\barmfM_g^\eps=T,
\]
we identify it as
\[
        G_{F_{\eta,M}}\simeq \cHom(L^{-1},T).
\]
The coefficient morphism
\[
        j_{F_{\eta,M}}:
        H^1\!\left(
        C,
        \cHom\!\left(
        \bigwedge\nolimits^2T_-X_\eta,T_+X_\eta
        \right)
        \right)
        \longrightarrow
        H^1(C,G_{F_{\eta,M}})
\]
is induced by the even differential
\[
        dF_{\eta,M,+}:T_+X_\eta=T_C
        \longrightarrow
        i^*T_+\barmfM_g^\eps=T .
\]
Under the identifications above, this is precisely the map induced by the
inclusion \(T_C\hookrightarrow T\).  At the level of coefficient sheaves, the
comparison is represented by
\[
\begin{tikzcd}[column sep=huge,row sep=large]
        \cHom\!\left(
        \bigwedge\nolimits^2T_-X_\eta,T_+X_\eta
        \right)
        \arrow[r]
        \arrow[d,"\simeq"']
        &
        \cHom\!\left(
        \bigwedge\nolimits^2T_-X_\eta,i^*T_+\barmfM_g^\eps
        \right)
        \arrow[d,"\simeq"]
        \\
        \cHom(L^{-1},T_C)
        \arrow[r]
        \arrow[d,"\simeq"']
        &
        \cHom(L^{-1},T)
        \\
        L^{-1}
        \arrow[ur,bend right=18,hookrightarrow,swap,"T_C\hookrightarrow T"]
        &
\end{tikzcd}
\]
where the curved arrow is the morphism obtained from the inclusion
\(T_C\hookrightarrow T\).  Passing to cohomology, this is exactly the map
\[
        J:
        H^1(C,L^{-1})
        \longrightarrow
        H^1\bigl(C,\cHom(L^{-1},T)\bigr)
\]
introduced in \Cref{sec:survival}.  Therefore
\[
        j_{F_{\eta,M}}\bigl(\omega_2(X_\eta)\bigr)=J(\eta).
\]
By construction of \(\eta\), this class is non-zero:
\[
        j_{F_{\eta,M}}\bigl(\omega_2(X_\eta)\bigr)=J(\eta)\ne0.
\]

We now apply \Cref{lem:functoriality} to
\[
        F_{\eta,M}:X_\eta\longrightarrow \barmfM_g^\eps.
\]
The lemma does not provide a morphism between the two obstruction groups
themselves.  Rather, it compares their images in the common cohomology group
\(
        H^1(C,G_{F_{\eta,M}}).
\)
Thus the relevant diagram is the cospan
\[
\begin{tikzcd}[column sep=huge,row sep=large]
        H^1\!\left(
        C,
        \cHom\!\left(
        \bigwedge\nolimits^2T_-X_\eta,T_+X_\eta
        \right)
        \right)
        \arrow[rr,"j_{F_{\eta,M}}"]
        &&
        H^1(C,G_{F_{\eta,M}})
        \\
        &&
        \\
        H^1\!\left(
        \barcS_g^\eps,
        \cHom\!\left(
        \bigwedge\nolimits^2T_-\barmfM_g^\eps,
        T_+\barmfM_g^\eps
        \right)
        \right)
        \arrow[uurr,bend right=18,"\iota_{F_{\eta,M}}"']
        &&
\end{tikzcd}
\]
and functoriality asserts the equality
\[
        j_{F_{\eta,M}}\bigl(\omega_2(X_\eta)\bigr)
        =
        \iota_{F_{\eta,M}}\bigl(\omega_2(\barmfM_g^\eps)\bigr)
        \quad\text{in }H^1(C,G_{F_{\eta,M}}).
\]
Since the left-hand side is \(J(\eta)\ne0\), we obtain
\[
        \iota_{F_{\eta,M}}\bigl(\omega_2(\barmfM_g^\eps)\bigr)
        =
        J(\eta)\ne0.
\]
In particular,
\[
        \omega_2(\barmfM_g^\eps)\ne0.
\]
By \Cref{prop:omega-stack}, non-vanishing of the primary obstruction prevents
projectedness.  Hence \(\barmfM_g^\eps\) is non-projected.

It remains only to obtain both parity components.  Elliptic curves admit theta
characteristics of both parities.  Since the genus-\(h\) spin curve \((C,L)\) is
even, the parity of the compact-type glued spin curve is the parity of the
elliptic tail:
\[
        \operatorname{par}\bigl((C,L,p)\cup_{p=q}(E,M,q)\bigr)
        =
        \operatorname{par}(L)+\operatorname{par}(M)
        =
        \operatorname{par}(M).
\]
Choosing \(M\) even gives the even compactified component, and choosing \(M\)
odd gives the odd compactified component.  Therefore both
\(\barmfM_g^+\) and \(\barmfM_g^-\) are non-projected for every \(g\ge3\).\end{proof}


\begin{thebibliography}{99}

\bibitem{AtiyahSpin}
M.~Atiyah,
\emph{Riemann surfaces and spin structures},
Ann. Sci. \'Ecole Norm. Sup. (4) \textbf{4} (1971), 47--62.

\bibitem{Bruzzo}
U.~Bruzzo, D.~Hern\'andez Ruip\'erez, 
\emph{Foundations of superstack theory}, 
arXiv:2505.19899.


\bibitem{CCC04}
L.~Caporaso, C.~Casagrande and M.~Cornalba,
\emph{Moduli of roots of line bundles on curves},
Trans. Amer. Math. Soc. \textbf{359} (2007), no.~8, 3733--3768.

\bibitem{CodViv} 
G.~Codogni and F.~Viviani, 
\emph{Moduli and periods of supersymmetric curves}, Adv. Theor. Math. Phys. {\bf 23} (2019), no.~2, 345--402.

\bibitem{Cornalba}
M.~Cornalba,
\emph{Moduli of curves and theta-characteristics},
in \emph{Lectures on Riemann Surfaces}, Trieste, 1987,
World Scientific, 1989, 560--589.


\bibitem{CorNoj}
M.~Corr\^{e}a, S.~Noja,
\emph{Formal moduli and the splitting theory of supermanifolds},
arXiv:2605.03166.


\bibitem{DWnp}
R.~Donagi and E.~Witten,
\emph{Supermoduli space is not projected},
in \emph{String-Math 2012}, Proc. Sympos. Pure Math. \textbf{90},
Amer. Math. Soc., Providence, RI, 2015, 19--71.

\bibitem{DWat}
R.~Donagi and E.~Witten,
\emph{Super Atiyah classes and obstructions to splitting of supermoduli space},
Pure Appl. Math. Q. \textbf{9} (2013), no.~4, 739--788.

\bibitem{FKP}
G. Felder, D.~A. Kazhdan and A. Polishchuk, \emph{The moduli space of stable supercurves and its canonical line bundle}, 
Amer. J. Math. {\bf 145} (2023), no.~6, 1777--1886.

\bibitem{FKP-boundary}
G. Felder, D. Kazhdan and A. Polishchuk,
\emph{Superperiods and superstring measure near the boundary of the moduli space of supercurves},
arXiv:2408.11136.


\bibitem{Green}
P.~Green,
\emph{On holomorphic graded manifolds},
Proc. Amer. Math. Soc. \textbf{85} (1982), no.~4, 587--590.


\bibitem{MumfordTheta}
D.~Mumford,
\emph{Theta characteristics of an algebraic curve},
Ann. Sci. \'Ecole Norm. Sup. (4) \textbf{4} (1971), 181--192.

\bibitem{WittenNotes}
E. Witten, \emph{Notes on super Riemann surfaces and their moduli}, 
Pure Appl. Math. Q. {\bf 15} (2019), no.~1, 57--211. 

\end{thebibliography}
\end{document}